\documentclass[12pt]{amsart}

\usepackage{amsmath, amssymb, amsthm, latexsym,nicefrac}
\usepackage[pagebackref,colorlinks,linkcolor=blue,citecolor=blue,urlcolor=blue,a4paper,hypertexnames=true,plainpages=false]{hyperref}
\usepackage[alphabetic]{amsrefs} 

\theoremstyle{plain}
\newtheorem{theorem}{Theorem}[section]
\newtheorem{lemma}[theorem]{Lemma}
\newtheorem{corollary}[theorem]{Corollary}
\newtheorem{proposition}[theorem]{Proposition}
\newtheorem{definition}[theorem]{Definition}

\theoremstyle{remark}
\newtheorem{remark}[theorem]{Remark}

\input{xy}
\xyoption{all}

\def\tr{{\rm tr}}
\def\cH{\mathcal H}
\def\Z{\mathbb Z}
\def\F{\mathbb F}
\def\C{\mathbb{C}}
\def\R{\mathbb{R}}
\def\N{\mathbb{N}}
\def\T{\mathbb{T}}
\def\fix{{\#_{\rm 1}}}

\title{Examples of hyperlinear groups without factorization property}
\author{Andreas Thom}
\address{Andreas Thom, Mathematisches Institut der Universit\"at G\"ottingen,
Bunsenstr. 3-5, D-37073 G\"ottingen, Germany}
\email{thom@uni-math.gwdg.de}
\urladdr{http://www.uni-math.gwdg.de/thom}

\begin{document}

\begin{abstract}
In this note we give an example of a group which is locally embeddable into finite groups (in particular it is initially subamenable, sofic and hence hyperlinear) but does not have Kirchberg's factorization property. This group provides also an example of a sofic Kazhdan group which is not residually finite, answering a question from \cite{MR2178069}. We also give an example of a group which is not initially subamenable but hyperlinear. Finally, we point out a mistake in \cite[Corollary 1.2 (v)$\Rightarrow$(i)]{MR1282231} and \cite[Corollary 7.3(iii)]{MR1218321} and provide an example of a group which does not have the factorization property and is still a subgroup of a connected finite-dimensional Lie group.
\end{abstract}

\maketitle

\section{Introduction}

In \cite{MR2072092}, N.\ Ozawa observed that there are no hyperlinear groups known which do not have E.\ Kirchberg's factorization property (see below for definitions). In this note, we give examples of groups with this behaviour. As a consequence, the maximal group $C^*$-algebra of such a group cannot have the local lifting property. To the best knowledge of the author, although this should be so for large classes of groups, no such group had been constructed before. Another example  yields a group which is hyperlinear but not \emph{initially subamenable}, a concept that goes back to M.\ Gromov, see \cite{gromov} for more information and Section \ref{secex} for a definition. The possibly weaker notion of being \emph{sofic} (which also goes back to \cite{gromov}, see also \cites{MR1803462,MR2178069,MR2220572} and Definition \ref{sofic-group}) is immediately connected with all that.
It has been asked by various people whether all hyperlinear groups are sofic and whether all sofic groups are initially subamenable. The first question has the flavour of A.\ Mal'cev's result that all finitely generated linear groups are residually finite. Concerning the second question, Gromov states in \cite[p.\ 157]{gromov} that "it may (?) happen" that a group is sofic but not initially subamenable. We do not have an answer, but our result gives that at least one of these interesting questions has a negative answer. 
\vspace{0.1cm}

The question here is about different qualities of matricial or combinatorial approximation to group laws. In the theory of von Neumann algebras one frequently encounters the following kind of approximation: Let $G$ be a finitely generated group with a fixed generating set $S \subset G$; this amounts to fixing a monoid surjection $w\colon F_S \to G$ from the free monoid on symbols $\{s,s^* \mid s\in S\}$, so that $w(s^*)=w(s)^*$. Given an integer $n$ and $\varepsilon >0$, one wants to find complex matrices $\{u_s\}_{s \in S}$ (of some arbitrary size $k$), such that for all words $t \in F_S$ of length $\leq n$ we have that $\tr(\overline{w}(t))$ is $\varepsilon$-close to zero whenever $w(t)$ is non-trivial and $\varepsilon$-close to one whenever $w(t)$ is trivial in $G$. Here, $\overline{w}$ denotes the natural monoid homomorphism $\overline{w} \colon F_S \to M_k \C$ satisfying $\overline{w}(s) = u_s$ and $\overline{w}(s^*)=u_s^*$. As usual, $\tr \colon M_k \C \to \C$ denotes the normalized trace on the algebra of complex $k \times k$-matrices. The matrices $\{u_s\}_{s \in S}$ are usually called $(\varepsilon,n)$-microstates. Sometimes, one assumes the existence of some universal bound on the operator norms of the matrices $u_s$, but this is not necessary.

A finitely generated group which admits $(\varepsilon,n)$-microstates for all $n \in \N$ and all $\varepsilon >0$ is said to be \emph{hyperlinear}. Whereas the study of microstates was initiated in work of D. Voiculescu on free probability theory, the name \emph{hyperlinear} goes back to F. R\u adulescu.
The class of hyperlinear groups includes all residually amenable groups and there is no group known, which is not hyperlinear. As a matter of fact, if $G$ is residually amenable, then the approximations above can be choosen to be induced (on the generators of $G$) by linear maps
$$\phi_k \colon C^* G \to M_{n_k} \C$$
which are unital, completely positive and satisfy
\begin{enumerate}
\item[(i)] $\lim_{k \to \infty} \|\phi_k(a)\phi_k(b) - \phi_k(ab)\|_2=0$, for all $a,b \in C^*G$, and
\item[(ii)] $\lim_{k \to \infty} |\tau(a) - \tr \circ \phi_k(a)| = 0$, for all $a \in C^*G$.
\end{enumerate}
Here, $C^*G$ denotes the maximal group $C^*$-algebra and $\tau\colon C^*G \to \C$ denotes the canonical trace. A linear map $\phi_k$ is said to be completely positive if the induced linear maps
$$1_{M_n\C} \otimes \phi_k \colon M_n \C \otimes_{\C} C^* G \to M_n \C \otimes_{\C} M_{n_k} \C$$
respect the cone of positive operators for all $n \geq 1$. The notion of complete positivity and complete boundedness appears naturally in the realm of $C^*$-algebras and is naturally embedded into the subject of operator space theory. For 
background on the theory of operator spaces we refer to the book by G.\ Pisier \cite{MR2006539}.

Following E.\ Kirchberg (\cite {MR1282231}), a group $G$ is said to have the \emph{factorization property} if it is hyperlinear and the approximation is induced by unital completely positive maps as described above. The factorization property can be defined in various ways and has been studied in detail in \cite{MR1218321, MR1282231}, see also \cite{MR2022373}.
Kirchberg proved the following remarkable and beautiful result:

\begin{theorem}[Kirchberg, see \cite{MR1282231}] \label{kirchberg}
For a Kazhdan group $G$, the following properties are equivalent: 
\begin{enumerate}
\item[(i)] $G$ has the factorization property.
\item[(ii)] $G$ is residually finite.
\item[(iii)] $G$ is isomorphic to a subgroup of the unitary group of the hyperfinite $II_1$- 
factor.
\end{enumerate}
\end{theorem}

Recall, a group is \emph{residually finite} if every non-trivial element can be mapped non-trivially to a finite group. For background on Kazhdan's property (T), see \cite{bekval}. We will use as a definition that $1$-cocycles into unitary representations are bounded. More precisely:
\begin{definition}
Let $G$ be a group. The group $G$ is said to have Kazhdan's property $(T)$ if for every unitary Hilbert space representation $\pi\colon G \to U(H)$ and every map 
$c\colon G \to H$ which satisfies $c(gh) = \pi(g) c(h) + c(g)$, there exists a constant $C \in \R$, such that $\|c(g)\| \leq C$, for all $g \in G$.
\end{definition}
C.\ Champetier (see \cite{MR1760424}) showed the existence of Kazhdan groups with no subgroups of finite index. These groups cannot have Kirchberg's factorization property and it remains an intriguing question to decide whether they can be hyperlinear.
Since these groups appear as inductive limits of hyperbolic groups, the question whether all hyperbolic groups are residually finite is immediately linked with this problem. 

There are two sources of simple groups with Kazhdan's property (T). Such groups appear for example as lattices in certain Kac-Moody groups, see \cite{caprem}. Much earlier, it was also shown by Gromov (\cite{MR919829}) that every hyperbolic group surjects onto a Tarski monster, i.e.\ every proper subgroup of this quotient is finite cyclic; in particular: this quotient group is simple and is a Kazhdan group if the hyperbolic group was a Kazhdan group.

\vspace{0.1cm}

In this note we describe explicitly a finitely generated hyperlinear (in fact locally embeddable into finite groups (LEF), hence sofic) Kazhdan group which is not residually finite and hence does not have the factorization property. In fact, we show that our example admits a surjective and non-injective endomorphism (i.e.\ it is non-hopfian) and use A.\ Mal'cev's famous result:

\begin{theorem}[Mal'cev, see \cite{MR0003420}] \label{malcev}
Every finitely generated residually finite group is hopfian.
\end{theorem}

Although we did not define all the notions used so far, we try to clarify the situation by drawing a diagram of known implications between the concepts involved. We will recall all relevant definitions along the way. Throughout, all groups are assumed to be finitely generated.
$$
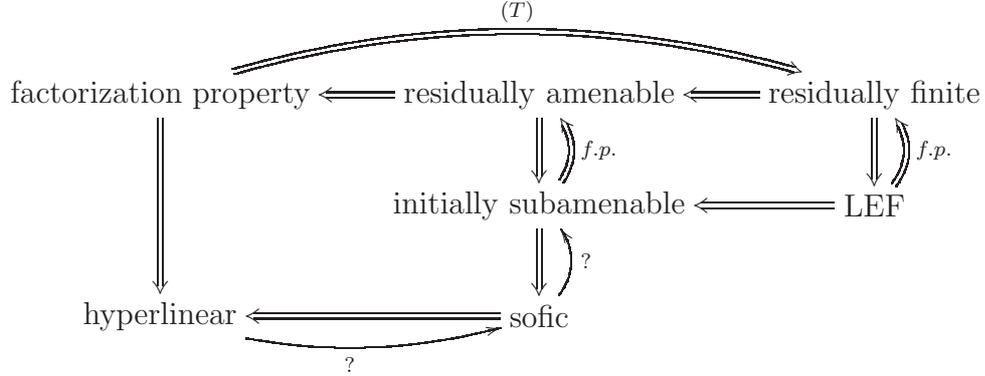
\xymatrix{
\mbox{factorization property} \ar@{=>}[dd]  \ar@/^2pc/@{=>}[rr]^{(T)} & \mbox{residually amenable} \ar@{=>}[d] \ar@{=>}[l] & \mbox{residually finite}\ar@{=>}[d] \ar@{=>}[l]\\
& \mbox{initially subamenable} \ar@{=>}[d]  \ar@/_1pc/@{=>}[u]_{f.p.}& \mbox{LEF} \ar@{=>}[l] \ar@/_1pc/@{=>}[u]_{f.p.} \\
\mbox{hyperlinear} \ar@/_1pc/@{->}[r]_{?} & \mbox{sofic} \ar@{=>}[l] \ar@/_1pc/@{->}[u]_{?}
}
$$
The decorations on the arrows indicates that additional assumptions are needed, \emph{f.p.}\ means \emph{finitely presented}.
The first main result of this article is:
\begin{theorem}
There exists a finitely generated Kazhdan group $G$ which is locally embeddable into finite groups (in particular sofic and hence hyperlinear) but does not have the factorization property.
\end{theorem}

This answers the question whether all hyperlinear groups could have the factorization property, see \cite[p.\ 524]{MR2072092}. As Ozawa points out (see \cite[p.\ 527]{MR2072092}), the maximal group $C^*$-algebra of $G$ cannot have the local lifting property. (Ozawa proved in \cite{MR2022373}, that there are groups whose maximal group $C^*$-algebra do not have the lifting property.) We do not go into the definition of the lifting property or the local lifting property and the relevance of this results, see \cite{MR2072092}. This example of a group also ends speculations on whether all sofic Kazhdan groups are necessarily residually finite, see \cite[Proposition $4.5$]{MR2178069}.

\vspace{0.1cm}

The second result is:

\begin{theorem}
There exists a hyperlinear Kazhdan group $K$ which is not initially subamenable.
\end{theorem}

Unfortunatelly, we cannot decide whether our example is sofic or not. In any case, we can conclude that either there exists a hyperlinear group which is not sofic, or there exists a sofic group which is not initially subamenable. Note that this group cannot have the factorization property.

\begin{remark} \label{disprove}
E.\ Kirchberg claims in \cite[Corollary 1.2]{MR1282231} that the three properties in Theorem \ref{kirchberg}
are equivalent to $G$ being isomorphic to a subgroup of an almost connected locally compact group. However, our second example $K$ of a group without factorization property arises as a subgroup of a connected finite dimensional Lie group and disproves this claim. The mistakes in the proof of Corollary $1.2$ of \cite{MR1282231} arises in the reference to the proof of Lemma $7.3$(iii) in \cite{MR1218321}. There, it is implicitly assumed that $G$ is unimodular and has a basis of almost conjugation invariant compact neighborhoods, see Remark \ref{kir2} for more details.
\end{remark}

I want to thank Alain Valette for pointing out a reference and contributing many helpful comments after reading the first draft of this paper.

\section{The first example}

\subsection{Non-hopfian Kazhdan groups}

We follow ideas of Y.\ de Cornulier \cite{MR2262894} in constructing non-hopfian groups with Kazhdan's property (T). Some of his ideas go back to work of H.\ Abels \cite{MR564423}. 

Let $R$ be a (unital) commutative ring and define $G_0(R)$ to 
be the following group of matrices:

$$G_0(R) = \left\{ \left. \left[ \begin{array}{ccccc} 
1 &a_{12}&a_{13} &a_{14}&a_{15} \\
0 &a_{22}&a_{23} &a_{24}&a_{25} \\
0 &a_{32}&a_{33} &a_{34}&a_{35} \\
0 &a_{42}&a_{43} &a_{44}&a_{45} \\
0 &0&0 &0&1
\end{array}  \right] \in SL_5(R) \right| a_{i,j} \in R \right\}$$

We will be mainly interested in $G' = G_0(\F_p[t,t^{-1}])$ for some prime $p$.
The centre $Z(G')$ of $G'$ consists of the elementary matrices $\{ e_{15}(a) \mid a \in \F_p[t,t^{-1}]\}$ and it is naturally isomorphic to the group
\begin{equation} \label{isom}
Z(G') \cong \oplus_{n \in \Z} \F_p t^n.
\end{equation} 
We denote by $C$ the subgroup
corresponding to $\oplus_{n \geq 0} \F_pt^n \subset Z(G')$ and set 
$G =G'/C$. Conjugation with the diagonal matrix $t \oplus 1 \oplus \cdots \oplus 1$
induces an automorphism of $G'$ and of its centre which shifts the index set under the isomorphism in Equation \ref{isom} and hence maps $C$ onto a proper subgroup of itself. As a consequence, the induced endomorphism on $G$ is surjective (since it is induced by a surjection) and not injective. Indeed, the kernel is isomorphic to $(\F_p,+)$ and sits as $\F_pt^{-1}$ in the upper right corner of $G$. We conclude that $G$ is non-hopfian.

\begin{lemma} \label{lem}
The group $G$ has Kazhdan's property (T). \end{lemma}
\begin{proof}
First of all, A.\ Suslin proved in \cite{MR0472792} that the group of elementary $3 \times 3$-matrices $EL_{3}(\F_p[t,t^{-1}])$ coincides with $SL_3(\F_p[t,t^{-1}]$ (see for example Proposition $5.4$ in \cite{MR2235330}). Secondly, Y.\ Shalom showed in \cite[Thm.\ 1.1]{MR2275645} that for any finitely generated and commutative ring $R$, the group $EL_n(R)$ is a Kazhdan group for 
$n \geq 2 + \dim\, R$, where \emph{dim} denotes the Krull dimension of the ring $R$.

Hence, since $\dim\, \F_p[t,t^{-1}] =1$, we get that $SL_3(\F_p[t,t^{-1}])$ is a Kazhdan group. Moreover, by \cite[Thm.\ 2.4]{MR2275645}, the pair $$(SL_2(\F_p[t,t^{-1}]) \ltimes \F_p[t,t^{-1}]^{\times 2},\F_p[t,t^{-1}]^{\times 2})$$ has the relative Kazhdan property; meaning that every 1-cocycle on the crossed product will be bounded on $\F_p[t,t^{-1}]^{\times 2}$.

Let now $\pi\colon G \to U(H)$ be a unitary representation of $G$ and $c\colon G \to H$ a $1$-cocycle. In order to show that $G$ has Kazhdan's property (T) it is sufficient to show that $c$ is bounded. Using the results above we see that $c$ is bounded on the copy of $SL_3(\F_p[t,t^{-1}])$ and bounded on the two copies of $\F_p[t,t^{-1}]^{\times 3}$. 
Since these subgroups generate $G$ boundedly, it follows that $c$ is bounded on $G$. This finishes the proof.
\end{proof}

\begin{remark}
It was pointed out to me by A.\ Valette, that there is a proof of Lemma \ref{lem} using more classical technology. Indeed, the Borel-Harish Chandra criterion applies to show that $G_0(\F_p[t,t^{-1}])$ is a $S$-arithmetic lattice in the locally compact group $\tilde G=G_0(\F_p[[t]]) \times G_0(\F_p[[t^{-1}]])$. Here, $\F_p[[t]]$ and $\F_p[[t^{-1}]]$ denote the Laurent power series rings in the variable $t$ and $t^{-1}$ respectively. Standard arguments along the lines of our proof now show that the locally compact group $\tilde G$ has property (T). This implies that also $G_0(\F_p[t,t^{-1}])$ (and hence $G$ being a quotient of $G_0(\F_p[t,t^{-1}])$) has property (T).
\end{remark}

The first consequence - which is not entirely obvious - is that $G$ is finitely generated. Since $G$ is not hopfian, we can now invoke Mal'cev's theorem (see Theorem \ref{malcev}) and get that $G$ is not residually finite. Hence, using Kirchberg's result (see Theorem \ref{kirchberg}), $G$ necessarily fails to have the factorization property. In order to retrieve the desired example, we show that the group $G$ is hyperlinear.

In fact, we will show that $G$ is \emph{locally embeddable into finite groups} (LEF), i.e. every finite piece of the multiplication table of $G$ can be found as a piece of the multiplication table of a finite group. This notion was studied by A.\ Vershik and E.\ Gordon in \cite{MR1458419}. Considering the left regular representation of the finite group, one easily sees that such groups are sofic (see \cite{MR2178069} and Definition \ref{sofic-group}) and hyperlinear, even with $\varepsilon=0$.

\begin{lemma}
For every finite subset $F \subset G$, there exists a finite group $K$ and an injective map 
$\phi \colon F^2 \to K$, such that
$$\phi(f_1f_2) = \phi(f_1)\phi(f_2),\quad \forall f_1,f_2 \in F.$$
\end{lemma}
\begin{proof}
Since $F$ is finite, there exists some integer $n \geq 1$, such that $t^k$ and $t^{-k}$ for $k \geq n$ are not involved in writing down the elements from $F$.

Consider $G_0(\F_p[\xi])$ where $\xi$ is a formal $6n$-th root of $1$ and consider
$K=G_0(\F_p[\xi])/C'$ where $C' = \oplus_{k=0}^{3n-1} \F_p \xi^k$. Clearly, since $\F_p[\xi]$ is a finite ring, the group $K$ is finite. It is also clear that the the natural identification of $F^2$ with a subset of $K$ which maps $t^k \mapsto \xi^k$ satisfies the required multiplicativity. This finishes the proof.
\end{proof}

\subsection{Some consequences for the theory of operator spaces}

Following ideas of M.\ Gromov, R.\ Grigorchuk (see \cite{MR764305}) has introduced the space of marked groups with $n$ generators. The elements in that space are (equivalence classes of) finitely generated groups $G$ with a fixed (ordered) generating set $S$ of cardinality $n$. Again, fixing an ordered generating set amounts to fixing a surjection $\phi\colon F_n \to G$. The distance is defined in terms of the kernel $N \subset F_n$ of $\phi$ as follows:
$$d((G,S),(G',S')) = 
\inf \{2^{-k} \mid k \in \N \colon N \cap B_{F_n}(k)  = N' \cap B_{F_n}(k)\},$$

where $B_{F_n}(k)$ denotes the ball of radius $k$ in $F_n$.

\vspace{0.1cm}

Let $(G,S)$ be a marked group. We denote by $C^* G|_{\leq k} \subset C^*G$ the operator system spanned by all group elements of length $\leq k$ with respect to the length induced by $S$. We consider a convergent sequence of marked groups $$ \lim_{i \to \infty} (G_i,S_i) = (G,S).$$
Denote by $p_{i,k}\colon C^*G|_{\leq k} \to C^*G_i|_{\leq k}$ the natural linear bijections (defined for large $i$) of operator systems induced by the identification of subsets of $G$ and $G_i$. It is natural to expect that convergence of marked groups is somehow resembled in the convergence of the associated operator systems. 
More precisely, one would like to decide whether the cb-norm of the maps $p_{i,k}\colon C^*G|_{\leq k} \to C^*G_i|_{\leq k}$
necessarily tends to one.
However, we show that this is not the case in general. We first have to recall a well-known theorem in the theory of operator spaces and conclude some corollary:

\begin{theorem}[Haagerup-Paulsen]
Let $\phi\colon A \to B(K)$ be a completely contractive map. Then there exists a Hilbert space $H$, a $*$-homomorphism $\pi\colon A \to B(H)$, 
and isometries $T,S \colon K \to H$, such that 
$$\phi(a)= S^*\pi(a) T.$$
\end{theorem}

It is well-known that every unital complete contraction is completely positive. The following consequence of the preceding theorem is an approximate version of this result, which is surely well-known to experts.

\begin{corollary} \label{coro}
Let $0 \leq \varepsilon < 1$.
Let $A$ be a unital $C^*$-algebra and let $\phi\colon A \to B(H)$ be a self-adjoint, unital, and completely bounded map with $\|\phi\|_{cb} \leq 1 + \varepsilon$. There exists a unital completely positive map $\psi\colon A \to B(H)$, such that
$$\|\phi - \psi\|_{cb} \leq 3 \sqrt{\varepsilon}.$$
\end{corollary}
\begin{proof}
Clearly, $\phi_{\varepsilon}(a) = (1 + \varepsilon)^{-1} \phi(a)$ is completely contractive
and hence, there exist isometries $T,S \colon K \to H$ and a unital $*$-homomorphism
$\pi\colon A \to B(H)$, such that
$$\phi_{\varepsilon}(a) = S^*\pi(a)T.$$

Since $S^*T = T^*S = \phi_{\varepsilon}(1) = (1+\varepsilon)^{-1}$, we compute:
$$(T^*-S^*)(T-S) = T^*T -S^*T - T^*S + S^*S = 2 - 2(1+\varepsilon)^{-1} \leq 2 \varepsilon.$$
We conclude that $\|T-S\| \leq \sqrt{2\varepsilon}$.
We can now simply set $\psi(a) = S^*\pi(a)S$, and compute
\begin{eqnarray*}
\|\psi(a) - \phi(a)\| &\leq& \|\psi(a) - \phi_{\varepsilon}(a)\| + \|\phi_{\varepsilon}(a) - \phi(a)\| \\
&\leq& \|S^*\pi(a)(T-S)\|+ (1-(1+\varepsilon)^{-1}))  \|\phi(a)\| \\
&\leq& (\sqrt{2 \varepsilon} + \varepsilon) \cdot \|a\| \\
&\leq& 3 \sqrt{\varepsilon} \cdot \|a\| 
\end{eqnarray*}
for $a \in A$. A similar computation for $a \in M_n\C \otimes_{\C} A$ yields that
$\|\psi -\phi\|_{cb} \leq 3 \sqrt{\varepsilon}$. This proves the claim, since $\psi(a)$ is unital completely positive.
\end{proof}

\begin{corollary}
There exists a convergent sequence of marked groups $G_i \to G$ and an integer $k \geq 1$, such that 
$$\liminf_{i \to \infty} \left\|p_{i,k}\colon C^* G|_{\leq k} \to C^* G_i|_{\leq k}\right\|_{cb} >1.$$
\end{corollary}
\begin{proof} We provided an example of a sequence of finite groups that converges in the space of marked groups to a group without the factorization property. If for every $k$, the cb-norm of the natural maps $p_{i,k} \colon C^*G|_{\leq k} \to C^*G_i|_{\leq k}$ would tend to one on a subsequence, then by Wittstock's extension theorem for completely bounded maps, there would exist linear maps $p'_{i} \colon C^*G \to B(\ell^2 G_i)$ (extending $C^*G|_{\leq k} \to C^*G_i|_{\leq k} \subset C^*G_i \subset B(\ell^2 G_i)$ with the same cb-norms. By Corollary \ref{coro}, a slight cb-norm perturbation gives unital completely positive maps close to that extensions. Since $p'_i$ 
is compatible with traces on group elements of length $\leq k$, we could conclude that the group $G$ has the factorization property. However, as we showed, this is not the case. Hence, there exists some $k$, for which the assertion fails.
\end{proof}
\begin{remark}
Unfortunately, we were unable to say anything about
$$\liminf_{i \to \infty} \left\|p^{-1}_{i,k}\colon C^* G_i|_{\leq k} \to C^* G|_{\leq k}\right\|_{cb}.$$ However, in many interesting cases of convergence of marked groups, the convergence is induced by surjective group homomorphisms $\phi_i \colon G_i \to G$. In this case the above maps $p^{-1}_{i,k}$ are induced by $*$-homomorphisms $\phi'_i\colon C^*G_i \to C^*G$ and hence are completely contractive.
\end{remark}

\section{The second example} \label{secex}

\subsection{An example of de Cornulier}

In this section we provide an example of a finitely presented group, which is hyperlinear but not initially subamenable. Recall, a group is said to be \emph{initially subamenable} if every finite piece of its multiplication table can be found as a piece of the multiplication table of an amenable group. Let $R$ be a (unital) commutative ring.
Following de Cornulier (see \cite{MR2262894}), we set:
$$K_0(R)=  \left\{ \left[ \begin{array}{ccccc} 
1 &a_{12}&a_{13} &a_{14} \\
0 &a_{22}&a_{23} &a_{24} \\
0 &0&a_{33} &a_{34} \\
0 &0&0 &1 \\
\end{array}  \right] \in SL_8(R) \left| \begin{array}{l} a_{2,2}, a_{3,3} \in SL_3(R) \\
a_{12}, a_{13}, a_{24}^t, a_{34}^t \in M_{13}(R) \\
a_{23} \in M_{33}(R) \\
a_{14} \in R
\end{array} \right.
\right\}$$

We set $K_1(R) = K_0(R)/R$ (where we identify $R$ with the centre of $K_0(R)$) and denote the image of $g \in K_0(R)$ in $K_1(R)$ by $\overline g$. There is a natural lift $g'$ of $\overline g$ in $K_0(R)$ and we write $g_z$ for the element $g(g')^{-1}$.
Associated with the natural split above there is a $2$-cocycle
$\alpha \colon K_1(R) \times K_1(R) \to R$ which classifies the extension:
$$0 \to R \to K_0(R) \to K_1(R) \to 0.$$

We will be mainly interested in $K_0(\Z[\nicefrac1p])$. The example has similar properties as our first example above. The centre of $K_0(\Z[\nicefrac1p])$ is isomorphic to $\Z[\nicefrac1p]$ and we set $K=K_0(\Z[\nicefrac1p])/\Z$, where we view $\Z$ as the natural subgroup $\Z \subset \Z[\nicefrac1p]$. It was shown in \cite{MR2262894} that $K$ is a finitely presented non-hopfian Kazhdan group. (The main advantage over the first example is that $K$ is finitely presented.) In particular, the group $K$ is not initially subamenable. Indeed, any finitely presented initially subamenable group is residually amenable. Being a Kazhdan group, a homomorphic image in an amenable group follows to be finite. Hence, each finitely presented initially subamenable Kazhdan group is necessarily residually finite. However, $K$ is not residually finite since it is finitely generated and non-hopfian.
Note that $K$ does not have the factorization property. 

\begin{remark} \label{kir2}
Note that $K$ arises as a dense subgroup of the connected finite dimensional Lie group $K_0(\R)/\Z$. This disproves the implication (v)$\Rightarrow$(i) of Corollary $1.2$ in \cite{MR1282231}, see Remark \ref{disprove}. Kirchberg's proof still implies that $K$ cannot arise as a  subgroup of a unimodular almost connected locally compact group $L$ which has a local basis of compact sets $\{S_n\}_{n\in \N}$ such that
$$\lim_{n \to \infty} \frac{\mu(gS_n \triangle S_n g)}{\mu(S_n)} \to 0 \qquad \forall g \in L,$$
where $\mu$ denotes the Haar measure on $L$. 
Eberhard Kirchberg informed us that it is sufficient to assume (as an additional assumption) that the locally compact group $L$ satisfies property (Z), which was also introduced in \cite{MR1218321}. However, Kirchberg shows in \cite{MR1218321} that $L$ has {\it not} property (Z) if $SL_2 \R \subset L$.
\end{remark}


\subsection{The group $K$ is hyperlinear}

We will now prove that $K$ is hyperlinear. The ultra-product $R_{\omega}$ and the concept of a field of von Neumann algebras will play a role in the proof. However, we do not want to give its definition and properties (see \cite{MR2072092} for more details) since we will provide a second and more elementary proof without using the theory of von Neumann algebras. 

\begin{proposition} \label{hyper}
The group $K$ is hyperlinear.
\end{proposition}

Given a group $G$ and a $2$-cocycle $\alpha\colon G \times G \to S^1$, we can consider the twisted group von Neumann algebra $L_{\alpha}[G/C]$. Formally, $L_{\alpha}[G]$ is the von Neumann algebra which is generated by the involutive algebra $\C_{\alpha}[G]$ in its GNS-representation with respect to the canonical trace.
Here, $\C_{\alpha}[G]$ is the $\C$-algebra with $\C$-linear basis $\{[g] \mid g\in G\}$ and multiplication $[g]\cdot [h] = \alpha(g,h) [gh]$. The involution is given by $[g]^*
=[g^{-1}]$ and the trace satisfies $\tau([g])= \delta_{g,e}$.

In our special situation, a group $K_0(\Z[\nicefrac 1p])$ is given together with a central subgroup $\Z$. Associated to a (set-theoretic) lift $\sigma\colon K=K_0(\Z[\nicefrac 1p])/\Z \to \Z$ there is a $2$-cocycle $\alpha\colon K \times K \to \Z$ which classifies the extension:
$$0 \to \Z \to K_0(\Z[\nicefrac 1p]) \to K \to 0.$$ Given any character $\beta\colon \Z \to S^1$, we consider the twisted group von Neumann algebras $L_{\beta \circ \alpha}[K]$.

\begin{lemma} \label{embed}
Let $G$ be a group and let $C$ be a central subgroup. The group $G$ is hyperlinear if and only if the twisted group von Neumann algebra $L_{\beta \circ \alpha}[G/C]$ embeds into $R_{\omega}$ for every character $\beta \in \hat C$. 
\end{lemma}
\begin{proof}
Since $C$ is central, $LG$ is naturally identified with a field of von Neumann algebras $L_{\beta \circ \alpha}[G/C]$ over the base $\hat C$, where we view the compact group $\hat C$ as a probability space with the Haar measure on $\hat C$. It is a standard fact that this implies that $LG$ is embeddable into $R_{\omega}$ (i.e.\ $G$ is hyperlinear) if and only if $L_{\beta \circ \alpha}[G/C]$ is embeddable into $R_{\omega}$ for almost all $\beta \in \hat C$. In particular, if $LG$ is embeddable into $R_{\omega}$, the set of $\beta$'s for which $L_{\beta \circ \alpha}[G/C]$ embeds is dense in $\hat C$. Since the character $\beta$ varies by definition continuously on $\hat C$, it is easily seen that the set of $\beta$'s for which an embedding exists is also closed. Indeed, if the sequence $(\beta_n)_{n \in \N}$ is convergent to $\beta \in \hat C$, then there exists an embedding:
$\iota\colon L_{\beta \circ \alpha}[G/C] \hookrightarrow \prod_{\omega} L_{\beta_n \circ\alpha}[G/C].$ Hence, the set of $\beta's$ for which $L_{\beta \circ \alpha}[G/C]$ embeds into $R_{\omega}$ is all of $\hat C$, if $LG$ embeds into $R_{\omega}$. This finishes the proof.
\end{proof}

\begin{proof}[First proof of Proposition \ref{hyper}]
Applying the ring homomorphism $r_q \colon \Z[\nicefrac1p] \to \Z/q\Z$ for $q$ prime to $p$, we get sufficiently many homomorphisms from $K_0(\Z[\nicefrac1p])$ to finite groups $K_0(\Z/q\Z)$ to see that $K_0(\Z[\nicefrac1p])$ is residually finite. In particular it is hyperlinear and hence (by Lemma \ref{embed}) the algebras $L_{ \beta \circ \alpha}[K_0(\Z[\nicefrac1p])/\Z]$ are embeddable into $R_{\omega}$ for \emph{all} $\beta \in \hat {\Z}$. In particular, $LK = L_{1 \circ \alpha}[K_0(\Z[\nicefrac1p])/\Z]$ is embeddable into $R_{\omega}$.
\end{proof}
\begin{remark}
The proof yields that every quotient of a hyperlinear group by a central subgroup is again hyperlinear. Note however, that a central quotient of a group with factorization property does not necessarily has Kirchberg's factorization property, as the example above shows.
\end{remark}

In order to make the argument in the first proof above more explicit, we will provide a concrete construction of microstates and try to explain the kind of problems one encounters when trying to prove that $K$ is sofic. Let us first recall the definition of the term \emph{sofic}.

\begin{definition} \label{sofic-group}
A group $G$ is called \emph{sofic} if for any real number $0<\epsilon<1$ and any finite
subset $F \subseteq G$ there exists a natural number $n \in \N$ and a map
$\phi\colon G\to S_n$ with the following properties:
\begin{enumerate}
\item[(i)]   $\fix \phi(g)\phi(h)\,\phi(gh)^{-1}\ge(1-\epsilon)n,\quad$ for any two elements $g,h\in F$,
\item[(ii)]   $\phi(e)=e \in S_n$,
\item[(iii)]   $\fix\phi(g)\le\epsilon n,\quad$ for any $e\ne g\in F$.
\end{enumerate}
\end{definition}

Here, $\fix \sigma$ denotes the number of fixed points of a permutation $\sigma \in S_n$. Sofic groups (with an equivalent definition in terms of Cayley graphs) were first studied by Gromov in  \cite{gromov} and later by B.\ Weiss in \cite{MR1803462} who also coined the name. Later, G.\ Elek and E.\ Szab{\'o} continued a more systematic study in \cite {MR2178069, MR2220572} and proved that sofic groups are hyperlinear.
We always think of $S_n$ as sitting inside $U(n)$ as permutation matrices. The number of fixed points of a permutation (normalized by the size of the set) is now identical with the normalized trace of the permutation matrix. Hence, showing that a group $G$ is sofic is the \emph{same} as finding microstates for $G$ within permutation matrices.
Although we are not able to show that the group $K$ is sofic we will provide microstates in the subgroup $\T^n \rtimes S_n \subset U(n)$.

\begin{lemma} Let $C$ be a residually finite abelian group. The set of characters $$\{\phi \colon C \to \T \mid \exists N \subset_{f.i.} C \colon \phi|_N = 1_N \}$$ is dense in the Pontrjagin dual $\widehat C$.
\end{lemma}
\begin{proof}
Since $C$ is residually finite, we have an injective map $$\theta\colon C \hookrightarrow  \prod_{N \subset_{f.i.} C} C/N.$$ Let $K$ be a closed subgroup of $\hat C$, such that we have a factorization $$\bigoplus_{N \subset_{f.i.} C} \widehat{C/N} \to K \stackrel{\nu}{\hookrightarrow} \hat C.$$
By duality, the inclusion $\theta$ factorizes over the surjection $\hat{\nu}\colon C \twoheadrightarrow \hat{K}.$ Hence, $\hat{\nu}$ is an isomorphism and so is 
$\nu$. This proves the claim.
\end{proof}

Let $q$ be an integer which is prime to $p$ and let $\pi_q\colon \Z[\nicefrac1p] \to \Z/q\Z$ the natural reduction modulo $q$.

\begin{corollary} \label{appr}
Let $k \geq 1$ be an integer and $\varepsilon>0$. There exists an integer $q \in \N$, prime to $p$, and characters $\beta_l\colon \Z/q\Z \to \T$, for $1 \leq l \leq p^k$ such that

$$|(\beta_l \circ \pi_q)(\nicefrac j{p^k}) - \exp(\nicefrac{2\pi i \cdot j \cdot l}{p^k})| < \varepsilon, \quad \forall 1 \leq j \leq p^k.$$
\end{corollary}

\begin{proof}[Second proof of Proposition \ref{hyper}] Let $S'$ be some finite generating set of $K$, let $S$ be some lift of $S'$ to $K_0(\Z[\nicefrac1p])$ and let $k$ be the highest power of $p$ involved in a denominator of $g_z$ for some $g \in (S \cup S^{-1})^n$. Let $\varepsilon >0$ be arbitrary and choose characters according to Corollary \ref{appr}. In addition, we may choose $q$ large enough so that $(S \cup S^{-1})^n$ is mapped injectively to $K_0(\Z/q\Z)$.

Denote by $\beta\colon \Z/q\Z \to \C^{p^k}$ the direct sum $\oplus_{l=1}^{p^k} \beta_l$. For $1 \leq l \leq p^k$, there exists a natural involution preserving homomorphism of rings $\phi_l \colon \C K_0(\Z[\nicefrac1p]) \to \C_{\beta_l\alpha} K_1(\Z/q\Z)$ given by:
$$\phi_l(g) = (\beta_l \circ \pi_q)(g_z) \cdot \pi_q(\overline g).$$ 
Here, $\C_{\beta_l} K_1(\Z/q\Z)$ denotes the twisted group algebra associated with the $2$-cocycle $\beta_l \alpha \colon K_1(\Z/q\Z) \times K_1(\Z/q\Z) \to S^1$.
The homomorphisms satisfies $\tau (\phi_l(g)) = (\beta_l \circ \pi_q)(g), \forall g \in \Z[\nicefrac 1p] =Z(K_0(\Z[\nicefrac1p]).$
We consider also the ring homomorphism
$$ \Phi = \oplus_{l=1}^{p^k} \phi_l \colon \C K_0(\Z[\nicefrac1p]) 
\to \oplus_{l=1}^{p^k} \C_{\beta_l \alpha} K_1(\Z/q\Z) =: \C_{\beta \alpha} K_1(\Z/q\Z).$$
The algebra $\C_{\beta\alpha} K_1(\Z/q\Z)$ is finite dimensional and carries a natural normalized trace, obtained by averaging the traces on the direct summands. We denote this trace by $\tau \colon \C_{\beta \alpha} K_1(\Z/q\Z) \to \C$.

We observe that
$|\tau(\Phi(\nicefrac j {p^k}))| < \varepsilon$, for all $1 \leq j < p^k$, since $\sum_{l=1}^{p^k} \exp(\nicefrac{2\pi i \cdot l}{p^k}) =0$, and similarly
$|\tau(\Phi(1)) -1| < \varepsilon$, for $1 \in \Z[\nicefrac1p]$.
The algebra $\C_{\beta \alpha} K_1(\Z/q\Z)$ acts naturally on the Hilbert space $\cH=\ell^2(K_1(\Z/q\Z)) \otimes \C^{p^k}$ via
\begin{equation} \label{formu}
(g_1, \cdots,g_{p^k}) \rhd (\delta_h \otimes \delta_j) = \beta_l(\alpha(g_l,h)) \cdot \delta_{g_lh} \otimes \delta_j.
\end{equation}
Here, $\alpha \colon K_1(\Z/q\Z) \times K_1(\Z/q\Z) \to \Z/q\Z$ is the natural $2$-cocycle as before. This presentation preserves the natural trace $\tau$ and combined with $\Phi$, it provides $(\varepsilon,n)$-microstates. Hence, $K$ is hyperlinear.
\end{proof}
\vspace{0.1cm}

In some sense this presentation is not far away from a permutation presentation, which would yield that $K$ is sofic. However, the matrices of our approximation are in the subgroup $\T^r \rtimes S_r \subset U(r)$ rather than $S_r \subset U(r)$, where $r= \dim\, \cH$. (This is obvious from Equation \ref{formu}.) Unfortunatelly, we were not able to remove the \emph{phase} in this approximation. 
Moreover, this kind of approximation is not well-suited for applications. In particular, it remains unclear whether direct finiteness of the group ring over some skew-field, the determinant conjecture (see \cite{MR2178069} and references therin) or the algebraic eigenvalue conjecture (see \cite{MR2417890}) can be proved for the group $K$.

\end{document}